\documentclass[11pt,reqno,oneside]{amsart}

\usepackage{appendix}

\usepackage{amsmath,amssymb,amsthm,amsfonts,mathrsfs,bbm}
\usepackage{fullpage}
\usepackage{amscd,bbm}
\usepackage[shortlabels]{enumitem}
\usepackage{xcolor}
\usepackage[colorlinks=true,
linkcolor=blue,
anchorcolor=blue,
citecolor=red
]{hyperref}

\allowdisplaybreaks 

\newtheorem{theorem}{Theorem}[section] 
\newtheorem{lemma}[theorem]{Lemma}

\newtheorem{proposition}[theorem]{Proposition}

\theoremstyle{definition}
\newtheorem{definition}[theorem]{Definition}

\theoremstyle{remark}

\newcommand{\Z}{\mathbb{Z}}

\newcommand{\C}{\mathbb{C}}

\newcommand{\GL}{\mathrm{GL}}
\newcommand{\PGL}{\mathrm{PGL}}

\newcommand{\SL}{\mathrm{SL}}

\newcommand{\Art}{\mathrm{Art}_F}

\newcommand{\on}{\operatorname}

\newcommand{\Ind}{\operatorname{Ind}}

\newcommand{\Irr}{\operatorname{Irr}}

\newcommand{\g}{\widehat{G}}
\newcommand{\ra}{\rightarrow}
\newcommand{\s}{\simeq}
\newcommand{\vp}{\varphi}
\newcommand{\der}{\on{der}}
\newcommand{\ab}{\on{ab}}

\newcommand{\Fr}{\on{Fr}}

\author{Kwangho Choiy}
\address{School of Mathematical and Statistical Sciences,
Southern Illinois University,
Carbondale, IL 62901-4408,
USA}
\email{kchoiy@siu.edu}

\author{Doyon Kim}
\address{Mathematical Institute of the University of Bonn, 53115 Bonn, Germany}
\email{kimdoyon@math.uni-bonn.de}
\author{Razan Taha}
\address{Department of Mathematics,
University of Michigan,
Ann Arbor, MI 48109-1043,
USA}
\email{tahar@umich.edu}
\author{Pan Yan}
\address{Department of Mathematics, The University of Arizona, Tucson, AZ 85721, USA}
\email{panyan@arizona.edu}

\makeatletter
\@namedef{subjclassname@2020}{\textup{}2020 Mathematics Subject Classification}
\makeatother
\date{\today}

\title{Component groups for non-supercuspidal $L$-parameters for $p$-adic $\mathrm{SL}_3$}


\begin{document}

\begin{abstract}
 We explicitly classify all the component groups associated to the non-supercuspidal, tempered $L$-parameters of $\mathrm{SL}_3(F)$ for a $p$-adic field $F$ of characteristic $0$ by direct case-by-case computations in $\mathrm{PGL}_3(\mathbb{C})$, following earlier work for $\mathrm{SL}_2$ by Labesse-Langlands and Shelstad.
\end{abstract}

\maketitle

\setcounter{tocdepth}{1}
\tableofcontents

\section{Introduction}
The internal structure of tempered $L$-packets is one of the intrinsic parts of the classical local Langlands conjecture for $p$-adic groups. While this conjecture has been well-formulated by Arthur \cite{art06} and later by Kaletha \cite{kalsim, kalrigd15-global} alongside the relationship to Arthur's conjectures, little is classified explicitly.

The main goal of this paper is to explicitly describe all the component groups in terms of matrices in $\PGL_3(\C)$ for the non-supercuspidal, tempered $L$-parameters of $\SL_3(F)$, where $F$ is a $p$-adic field of characteristic zero. The approach is to follow Labesse-Langlands \cite{ll79} and Shelstad \cite{shel79} and compute the centralizers of the images of the $L$-parameters directly in $\PGL_3(\C)$.
We should remark that this paper omits the supercuspidal $L$-parameters due to lack of the full classification of supercuspidal $L$-parameters of $\GL_3(F)$.

To be precise, based on \cite{gk82, hs11, abps13}, we denote by $F$ a $p$-adic field of characteristic zero and set $G(F)=\SL_3(F)$ for the introduction. We refer the reader to Section \ref{depcomponentgp} for the relevant detailed definitions in what follows. Note that the $L$-group of $G(F)$ is $^L G =\widehat{G} = \PGL_3(\C)$. 

Given a tempered $L$-parameter $[\vp]$ of $G(F)$, 
we have
\[
{Z}(\widehat{G}_{sc}) = \mathbf{\mu_3}(\C),~~ Z(\widehat{G})^{\Gamma} = 1,~\text{ and } \widehat Z_{\vp, sc} = \mathbf{\mu_3}(\C) / (\mathbf{\mu_3}(\C) \cap S_{[\vp], sc}^{\circ}).
\] 
Now, we consider the exact sequence
\[
1 \longrightarrow \C^{\times} \longrightarrow {\GL}_3(\C) \overset{pr}{\longrightarrow} {\PGL}_3(\C) = \widehat{G} \longrightarrow 1.
\]
Due to \cite[Th\'{e}or\`{e}m 8.1]{la85} (also, see \cite{weil74, he80}), we have an 
$L$-parameter $\vp$ for $\GL_3(F)$
\[
\vp : W_F \times {\SL}_2(\C) \rightarrow {\GL}_3(\C)
\]
such that ${pr} \circ \vp=[\vp]$.
By the local Langlands conjecture for $\GL_n$ \cite{ht01, he00, scholze13}, 
we have a unique irreducible representation $\sigma \in \Irr(\GL_3(F))$ associated to the $L$-parameter $\vp$ and the $L$-packet $\Pi_{[\vp]}(G)$ is the set consisting of all the irreducible, inequivalent constituents in the restriction of $\sigma$ to $G(F)$.
Thus, we have the following diagram
\begin{equation} 
\label{a diagram}
\begin{CD}
1 @>>> \mathbf{\mu_3}(\C) @>>> {\SL_3}(\C) @>>> {\PGL_3}(\C) @>>> 1 \\
@.{||} @.{\cup} @.{\cup} @.\\
1 @>>> Z(\g_{sc}) @>>> S_{\varphi,sc}(\g) @>>> S_{\varphi}(\g) @>>> 1 \quad .
\end{CD} 
\end{equation}
We consider a character $\zeta_{\SL_3}$ of ${Z}(\widehat{G}_{sc})$ 
which corresponds to $\SL_3$
via the Kottwitz isomorphism, and it is clear that $\zeta_{\SL_3} = \mathbbm{1}$, so that the internal structure of $L$-packets of $\SL_3(F)$ is in bijection with the set of equivalence classes of irreducible representations of the connected components in $\PGL_3(\C)$ of $L$-parameters of $\SL_3(F)$. 

Our main result in this framework of the internal structure of $L$-packets is to explicitly describe all the possible component groups 
\[
\mathcal{S}_{[\vp]}(\SL_3):=\pi_0\left(S_{[\varphi]}(\PGL_3(\C))\right)
\]
for non-supercuspidal, tempered $L$-parameters $[\varphi]$ of $\SL_3(F)$. 
We divide the analysis into three disjoint cases: unitary principal series $L$-parameters in Subsection \ref{subsec-ups}, non-supercuspidal discrete $L$-parameter in Subsection \ref{subsec-st}, and non-discrete tempered $L$-parameters in Subsection \ref{subsec-nondisc}. All such cases are computed directly in $\PGL_3(\C)$. 
\begin{theorem} \label{mainthmintro}
    Let $[\vp]$ be a non-supercuspidal, tempered $L$-parameter of $\SL_3(F)$, where $F$ is a $p$-adic field of characteristic $0$. Then the following holds:
    \begin{enumerate}
        \item[(i)] (Proposition \ref{ups}) If $\vp$ is an $L$-parameter of $\GL_3(F)$ corresponding to an irreducible, unitary principal series representation, then either 
$\mathcal{S}_{[\vp]}\cong 1$ or $\mathcal{S}_{[\vp]}\cong \Z/3\Z$.
        \item[(ii)] (Proposition \ref{pro-st}) If $\vp$ corresponds to the Steinberg representation twisted by a unitary character $\chi$ on $F^\times$, then $\mathcal{S}_{[\vp]}$ is trivial.
        \item[(iii)] (Propositions \ref{pro-dihedral}, \ref{pro-exceptional}, and \ref{pro-st-char}) If $\vp$ corresponds to $\displaystyle \Ind_{P_{2,1}}^{\GL_3(F)}(\sigma \otimes \chi)$, where $P_{2,1}$ is the standard parabolic subgroup whose Levi subgroup is $\GL_2(F) \times \GL_1(F),$ $\sigma$ is a discrete series representation of $\GL_2(F)$ and $\chi$ is a unitary character of $F^\times$, then $\mathcal{S}_{[\vp]}$ is trivial.
    \end{enumerate}
\end{theorem}
Our result shows that the component groups are all trivial,
except for a unitary principal series $L$-parameter, in which case $\Z/3\Z$ appears. 
As mentioned above, our main idea is to follow the approach of Labesse-Langlands \cite{ll79} and Shelstad \cite{shel79}. A case-by-case analysis with detailed computations is provided in Section \ref{mainresultsection}.

We remark that our result on component groups provides explicit descriptions for their various finite groups such as Knapp-Stein $R$-groups for reducibilities of parabolic inductions, stabilizers in Weyl groups and character groups, appearing in some earlier work, for instance, \cite[Proposition 1.9]{sh83}, \cite[Theorems 2.4 and 3.4]{go94sl}, and \cite[Theorem 4.3]{gk82}.
We also note that our matrix computations above are expected to apply to the centralizers of certain $L$-parameters of $\SL_n$ for $n >3$ (see Section \ref{closingremarks}).

\subsection*{Acknowledgements}
This project was started at the Rethinking Number Theory Workshop in Summer 2022 and we thank the organizers and Neelima Borade. The first author was supported by a gift from the Simons Foundation (\#840755). The second author was partially supported by ERC Advanced Grant 101054336 and Germany’s Excellence Strategy grant EXC-2047/1 - 390685813. The fourth author was partially supported by an AMS-Simons Travel Grant. 

\section{Background}
\subsection{The Weil Group}
Let $F$ be a $p$-adic field of characteristic zero with ring of integers $\mathfrak{o}$. For a given uniformizer $\varpi\in \mathfrak{o}$, let $\mathfrak{p}=\varpi \mathfrak{o}$ be the prime ideal of $\mathfrak{o}$ and denote the cardinality of the residue field $\mathfrak{f}=\mathfrak{o}/\varpi \mathfrak{o}$ by $q$, which is a power of $p$. Fix an algebraic closure $\overline{F}$ of $F$. Let $F^{\mathrm{unr}}$ be the maximal unramified extension of $F$ in $\overline{F}$ and $\overline{\mathfrak{o}}$ be the integral closure of $\mathfrak{o}$ in $F^{\mathrm{unr}}$. Then $\overline{\mathfrak{f}}:=\overline{\mathfrak{o}}/\varpi \overline{\mathfrak{o}}$ is an algebraic closure of $\mathfrak{f}$.

The Weil group $W_F$ is the dense subgroup of $\Gamma=\mathrm{Gal}(\overline{F}/F)$ consisting of elements which induce on $\overline{\mathfrak{f}}$ an integer power of the automorphism $x\mapsto x^q$. Fix a geometric Frobenius element $\Fr\in W_F$, which satisfies $\Fr(x^q)=x$ on $\overline{\mathfrak{f}}$. Then $W_F$ is a semidirect product $W_F=\langle \Fr\rangle \ltimes I_F$, where $I_F$ is the inertia subgroup. Denote by $W_F^{\der}$ the closure of the commutator subgroup of $W_F$, and write $W_F^{\ab}=W_F/W_F^{\der}$. By local class field theory, 
there is a canonical continuous group homomorphism
\begin{equation*}
\Art: W_F\to F^\times	
\end{equation*}
which induces an isomorphism $F^\times \cong W_F^{\ab}$, normalized so that $\varpi$ maps to the class of the geometric Frobenius $\Fr$ in $W_F^{\ab}$. The normalized valuation on $F^\times$ gives a homomorphism $|\cdot|: W_F\to q^{\Z}$ which satisfies $\ker |\cdot|=I_F$ and $|\Fr|=q^{-1}$. 
\subsection{Langlands Parameters}
 A Weil-Deligne representation of $W_F$ is a triple $(\rho, \mathcal{G}, N)$, where $\mathcal{G}$ is a complex Lie group whose identity component $\mathcal{G}^\circ$ is reductive, $\rho:W_F\to \mathcal{G}$ is a homomorphism which is continuous on $I_F$ such that $\rho(\Fr)$ is semisimple in $\mathcal{G}$, and $N$ is a nilpotent element in the Lie algebra $\mathfrak{g}$ of $\mathcal{G}$ such that $\rho(w)N\rho(w)^{-1}=|w|\cdot N$ for all $w\in W_F$. 
Two representations $(\rho, \mathcal{G}, N)$ and $(\rho^\prime, \mathcal{G}^\prime, N^\prime)$ are equivalent if there is an element $g\in \mathcal{G}^\circ$ such that $\rho^\prime=g\rho g^{-1}$ and $N^\prime=gNg^{-1}$. We denote the Weil-Deligne representation by $(\rho,N)$ when the group $\mathcal{G}$ is understood.
 
For a quasi-split connected reductive algebraic group $\mathbb{G}$ defined over $F$, let $G=\mathbb{G}(F)$ denote the $F$ points of $\mathbb{G}$. The $L$-group of $G$ is given by ${}^LG=\widehat{G}\rtimes W_F$, where $\widehat{G}$ is the complex Langlands dual group of $G$.

A Langlands parameter, or $L$-parameter, is a $\hat{G}$-conjugacy class of $G$-relevant $L$-homomorphisms
\begin{equation*}
 \varphi: W_F\times \SL_2(\C)\to {}^LG. 
\end{equation*}
In particular, $\varphi$ is trivial on an open subgroup of $I_F$, $\varphi(\Fr)$ is semisimple, and $\varphi|_{\SL_2(\C)}$ is a homomorphism of algebraic groups over $\C$. Moreover, any parabolic subgroup ${}^LP$ of ${}^LG$ that contains $\textrm{Im}(\vp)$ corresponds to a parabolic subgroup $P$ of $G$ that is defined over $F$. Let $\Phi(G)$ be the set of $L$-parameters of $G$. 

There is a bijection $(\rho, {}^LG, N)\longleftrightarrow \varphi$ between the equivalence classes of Weil-Deligne representations and the $\hat{G}$-conjugacy classes of homomorphisms $\varphi:W_F\times\SL_2(\C)\to {}^LG$ (\cite[Proposition 2.2]{gr10}). Given an $L$-parameter $\varphi:W_F\times\SL_2(\C)\to {}^LG$, we obtain a Weil-Deligne representation $(\rho, {}^L G, N)$ by setting, for any $w\in W_F$,
\begin{equation*}
\rho(w)=\varphi\left(w,\begin{pmatrix}
 |w|^{1/2} & 0 \\ 0 & |w|^{-1/2}
 \end{pmatrix}\right) \quad \text{and} \quad \displaystyle e^N=\varphi\left(1,\begin{pmatrix}
 1&1\\0&1
 \end{pmatrix}\right). 	
\end{equation*}
We often use the term $L$-parameter to refer to both $\varphi$ and the representation $(\rho,\mathcal{G},N)$ which corresponds to $\varphi$ under this bijection. 
\subsection{Local Langlands Correspondence for $\GL_n(F)$}
The local Langlands correspondence predicts a canonical partition of the set $\mathrm{Irr}(G)$ of isomorphism classes of irreducible, smooth, complex representations of $G$ as
\begin{equation*}
\Irr(G)=\bigcup_{\varphi\in \Phi(G)}\Pi_\varphi.
\end{equation*}
Here, the set $\Pi_{\varphi}$ is called the $L$-packet of $\varphi$. Moreover, the induced map
 \begin{equation} \label{eq-LLC-GLn}
 \mathcal{L}_G: \mathrm{Irr}(G) \to \Phi(G)
 \end{equation} 
is surjective, finite-to-one, and $\mathcal{L}_G(\Pi_{\varphi})=\varphi$. 

The local Langlands correspondence for $\GL_n(F)$, established in \cite{ht01, he00, scholze13}, says that the map \eqref{eq-LLC-GLn} is a bijection. Hence, all the $L$-packets $\Pi_\varphi \in \mathrm{Irr}({\GL_n})$ are singletons. Note that, since $\GL_n(F)$ is split over $F$, the $W_F$ action on $\widehat{G}$ is trivial, and we may consider the $L$-parameters to be $L$-homomorphisms $ \varphi:W_F\times \SL_2(\C)\to \GL_n(\C)$. We give an explicit description of the local Langlands correspondence for $\GL_2(F)$ and a partial description for $\GL_3(F)$ in Section \ref{sec-LLC-GL2-3}.

\subsection{Component Groups and the Internal Structure of Tempered $L$-packets} \label{depcomponentgp}
An $L$-parameter $\vp$ does not, in general, parametrize the irreducible representations in the corresponding $L$-packet $\Pi_\vp$. It is conjectured that an enhancement of the $L$-parameter with an irreducible representation of the component group of $\vp$ is the appropriate object to study. We refer the reader to \cite{art06, kalsim, kalrigd15-global} where one may observe two formulations and their relations (c.f. \cite[Section 4.1]{choiymulti}).

Following Arthur's formulation \cite{art06, art12}, let $Z(\g)$ be the center of $\g$, $\g_{ad}$ be the adjoint group $\g / Z (\g)$ of $\g$, and $\g _{sc}$ be the simply connected cover of the derived group $\g _{der}$ of $\g$. 

Let $Z_{\g}(\mathrm{Im}(\vp))$ denote the centralizer in $\g$ of the image of an $L$-parameter $\vp$. Define the complex reductive group
\begin{equation*}
 S_{\varphi}(\g)=Z_{\g}(\mathrm{Im}(\vp)) / Z(\g)^{\Gamma},
\end{equation*}
where $Z(\g)^{\Gamma}$ is the group of invariants of $Z(\g)$ under the action of the Galois group $\Gamma$, and note that $S_{\varphi}(\g)$ is a subgroup of $\g_{ad}$. Pulling back $S_{\varphi}(\g)$ under the quotient map $\g _{sc} \to \g_{ad}$ and considering it as a subgroup of $\g_{sc}$, we get $S_{\varphi,sc}(\g)$. Hence, we have the short exact sequence 
\begin{equation*}
 1 \ra Z(\g_{sc}) \ra S_{\varphi,sc}(\g) \ra S_{\varphi}(\g) \ra 1.
\end{equation*}
Define the component group of $S_{\varphi,sc}(\g)$ by 
\begin{equation*}
 \mathcal{S}_\vp=\pi_0(S_{\varphi,sc}(\g))=S_{\varphi,sc}(\g) / S_{\varphi,sc}(\g)^\circ
\end{equation*}
Then there exists a short exact sequence 
\begin{equation*}
 1 \ra \widehat{Z}_{\varphi,sc}(G) \ra \mathcal{S}_\vp \ra \mathcal{S}_{\varphi}(\g) \ra 1,
\end{equation*}
where $\widehat{Z}_{\varphi,sc}(G) = Z(\g_{sc}) / (Z(\g _{sc}) \cap S_{\varphi}(\g)^o)$ and $\mathcal{S}_\vp(\g) = \pi_0(S_{\varphi}(\g))$. 

Consider an inner twist corresponding to $G^* \ra G$ for a quasi-split inner form $G^*$ of $G$. Consider a character $\zeta_G$ of $Z(\g_{sc})$ whose restriction to $Z(\g_{sc})^{\Gamma}$ gives the class of the inner form $G$ of $G^*$ by the Kottwiz isomorphism (\cite[Theorem 1.2]{kot86}). 

If $\varphi$ is a tempered $L$-parameter of $G$, it is conjectured (refer to \cite[Section 3]{art06} and \cite[Section 9]{art12}) that there exists a one-to-one correspondence \begin{equation*}
 \Pi_{\varphi}(G) \longleftrightarrow \Irr(S_{\varphi,sc}(\g), \zeta_G),
\end{equation*}
where $\Irr(S_{\vp, sc}(\g), \zeta_G)$ denotes the set of irreducible representations of $S_{\varphi,sc}(\g)$ that are equivariant under the pullback of $\zeta_G$ to $\widehat{Z}_{\varphi, sc}(G)$. We remark that there is another formulation by Kaletha in terms of a different finite group and that both formulations are equivalent \cite{kalsim, kalrigd15-global}.

Note that, since $\SL_3(F)$ is split and simply connected, and since $Z(\PGL_3(\C))$ is trivial, the component group of an $L$-parameter $\vp:W_F\times \SL_2(\C)\to \PGL_3(\C)$ is given by the simpler equation 
\begin{equation*}
 \mathcal{S}_\vp=Z_{\PGL_3(\C)}(\mathrm{Im}(\vp)) / Z_{\PGL_3(\C)}(\mathrm{Im}(\vp))^\circ.
\end{equation*}

\subsection{Notation}
Throughout this paper, $B_n$ denotes the upper triangular Borel subgroup of $\GL_n$ and $P_{n_1,\cdots,n_k}$ denotes the standard parabolic subgroup of $\GL_n$ corresponding to the partition $n=n_1+\cdots+n_k$. For any matrix $A$, we write $a_{ij}$ for the entry in its $i^{th}$ row and $j^{th}$ column. Moreover, if $A\in\GL_n(\C)$, we write $[A]$ for its projection into $\PGL_n(\C)$. Similarly, we write $[\vp]= {pr} \circ \vp$ for the projection of an $L$-parameter $\vp:W_F\times \SL_2(\C)\to \GL_n(\C)$ into $\PGL_n(\C)$. 

Given a Levi decomposition $P=MN$ of any parabolic subgroup $P$ of $G$ and a representation $\sigma \in \Irr(M)$, $\Ind_{P(F)}^{G(F)}\sigma$ is the normalized parabolic induction with trivial $N(F)$-action.

\section{$L$-parameters of $\GL_2(F)$ and $\GL_3(F)$} \label{sec-LLC-GL2-3}
\subsection{Non-Supercuspidal $L$-parameters of $\GL_2(F)$} Let $\chi, \chi_1$, and $\chi_2$ be characters of $F^\times$. 
When $\frac{\chi_1}{\chi_2}\not=|\cdot|^{\pm 1}$, the induced representation $\Ind_{B_2(F)}^{\GL_2(F)}(\chi_1\otimes\chi_2)$ is an irreducible principal series representation, which corresponds under the bijection \eqref{eq-LLC-GLn} to the $L$-parameter 
\begin{equation*}
 \left((\chi_1\circ \Art)\oplus (\chi_2\circ \Art),\begin{pmatrix} 0 & 0 \\ 0 & 0\end{pmatrix}\right).
\end{equation*}
If $\frac{\chi_1}{\chi_2}=|\cdot|$, the principal series representation $\Ind_{B_2(F)}^{\GL_2(F)}(\chi_1\otimes\chi_2)=\Ind_{B_2(F)}^{\GL_2(F)}(\chi|\cdot|^{\frac{1}{2}}\otimes\chi|\cdot|^{-\frac{1}{2}})$, with $\chi=\chi_2 |\cdot|^{\frac{1}{2}}$, is reducible and has an irreducible quotient isomorphic to the one-dimensional representation $\chi \circ \det$, which corresponds under the bijection \eqref{eq-LLC-GLn} to the $L$-parameter
\begin{equation*}
 \left((\chi\circ \Art)\oplus (\chi\circ \Art),\begin{pmatrix} 0 & 0 \\ 0 & 0\end{pmatrix}\right).
\end{equation*}
If $\frac{\chi_1}{\chi_2}=|\cdot|^{-1}$, the induced representation $\Ind_{B_2(F)}^{\GL_2(F)}(\chi_1\otimes\chi_2)=\Ind_{B_2(F)}^{\GL_2(F)}(\chi|\cdot|^{-\frac{1}{2}}\otimes\chi|\cdot|^{\frac{1}{2}})$, with $\chi=\chi_1 |\cdot|^{\frac{1}{2}}$, is reducible and has an irreducible quotient isomorphic to the $\chi$-twisted Steinberg representation $\mathrm{St}_2(\chi)$, which corresponds to the $L$-parameter 
\begin{equation*}
 \left( (\chi\circ \Art)\oplus (\chi\circ \Art), \begin{pmatrix} 0& 1\\ 0 &0\end{pmatrix}\right).
\end{equation*}
We refer the reader to \cite{ku94} for more details.

\subsection{Supercuspidal $L$-parameters of $\GL_2(F)$} The supercuspidal representations of $\GL_2(F)$ correspond under the bijection \eqref{eq-LLC-GLn} to the $L$-parameter 
\begin{equation*}
 \left(\rho,\begin{pmatrix} 0 & 0 \\ 0 & 0\end{pmatrix}\right),
\end{equation*}
where $\rho$ is an irreducible, two-dimensional representation of $W_F$. Furthermore, supercuspidal $L$-parameters are classified by the image of $\vp:W_F\times \SL_2(\C) \to \GL_2(\C)$ projected to $\PGL_2(\C)$, following \cite{tate79, ht01, he00, scholze13} (see also \cite{bh06}, Sections 40, 41, and 42). We shall use these $\GL_2$ supercuspidal parameters in Subsections \ref{dihedralsubsection} and \ref{exceptionalparametersection}.
\begin{definition}
An irreducible $L$-parameter $\vp \in \Phi({\GL}_2)$ 
is called \textit{exceptional} (or \textit{primitive}) if the restriction $\vp|_{W_F}$ is not of the form $\Ind_{W_E}^{W_F} \theta$ for some quadratic extension $E$ over $F$. 
\end{definition}

\begin{definition}
An irreducible $L$-parameter $\vp \in \Phi({\GL}_2)$ is called \textit{dihedral} with respect to a quadratic extension $E$ over $F$ if 
$\vp|_{W_F} \s \Ind_{W_E}^{W_F} \theta$, or equivalently if 
$(\vp|_{W_F}) \otimes \omega_{E/F} \s (\vp|_{W_F})$, for a quadratic character $\omega_{E/F}$ corresponding to $E/F$ via local class field theory.
\end{definition} 
\begin{proposition} \label{prop-supercuspidal}
The image of a dihedral $L$-parameter $\vp$ in $\PGL_2(\C)$ is the dihedral group of order 4, which is generated (up to conjugacy) by the matrices
\begin{equation*}
 \begin{pmatrix}
 0 & 1 \\ 1 & 0 \end{pmatrix} \, \text{ and }\, \begin{pmatrix}
 -1 & 0 \\ 0 & 1 \end{pmatrix}.
\end{equation*}
The image of an exceptional $L$-parameter $\vp$ in $\PGL_2(\C)$ can either be the tetrahedral group $A_4$ or the octahedral group $S_4$. 
\end{proposition}
\subsection{Tempered, Non-Supercuspidal $L$-parameters for $\GL_3(F)$} \label{sec-LLC-GL3}
The irreducible, tempered, non-supercuspidal representations of $\GL_3(F)$ arise by parabolic induction from (essentially) discrete series representations of the Levi subgroups of $\GL_3$ (refer to \cite[Section 1.2]{ku94} for more details). Since the only partitions of $n=3$ are $n=1+1+1$ and $n=2+1$, the only possibilities are inducing characters of $F^\times$ or discrete series representations of $\GL_2(F)$. As in the $\GL_2(F)$ case, the irreducible, principal series representations of $\GL_3(F)$ are of the form
\begin{equation*}
\lambda=\Ind_{P_{1,1,1}(F)}^{\GL_3(F)}(\chi_1 \otimes \chi_2 \otimes \chi_3).
\end{equation*}
Here, $\chi_i$ with $i=1,2,3$ are unitary characters of $F^\times$, and $\chi_i\chi_j^{-1}\neq |\cdot|$ for all $1\leq i,j\leq 3$. These representations correspond to $L$-parameters of the form 
\begin{equation}\label{Langlands parameter 1}
 \left((\chi_1\circ \Art)\oplus (\chi_2\circ \Art)\oplus (\chi_3\circ \Art),\left(\begin{smallmatrix} 0 & 0 & 0\\ 0&0 & 0\\ 0&0 & 0\end{smallmatrix}\right)\right).
\end{equation}
When the representation $\lambda$ is reducible, the Steinberg representation twisted by a unitary character $\chi$, $\mathrm{St}_3(\chi)$, is realized as the unique, irreducible quotient of $\lambda$, and it corresponds to the $L$-parameter
\begin{equation}\label{Langlands parameter 2}
 \left((\chi\circ \Art)\oplus (\mathbbm{1}\circ\Art)\oplus (\chi^{-1}\circ \Art),\left(\begin{smallmatrix} 0 & 1 & 0\\ 0&0 & 1\\ 0&0 & 0\end{smallmatrix}\right)\right).
\end{equation}
Tempered, non-supercuspidal representations of $\GL_3(F)$ arising from the partition $n=2+1$ are of the form 
\begin{equation*}
\Ind_{P_{2,1}(F)}^{\GL_3(F)}(\sigma \otimes \chi)
\end{equation*} 
where $\sigma$ is a discrete series representation of $\GL_2(F)$ and $\chi$ is a unitary character of $F^\times$. When $\sigma$ is supercuspidal, the corresponding $L$-parameter is of the form
\begin{equation*}
 \left(\rho\oplus(\chi\circ \Art),\left(\begin{smallmatrix} 0 & 0 & 0 \\ 0 & 0&0\\ 0 & 0&0\end{smallmatrix}\right)\right),
\end{equation*}
where $\rho$ is an irreducible, two-dimensional representation of $W_F$. When $\sigma$ is the unitary twist of the Steinberg representation $\mathrm{St}_2(\mu)$, the corresponding $L$-parameter is of the form
\begin{equation} \label{stparameter}
 \left((\mu\circ \Art)\oplus (\mu\circ \Art)\oplus(\chi\circ \Art),\left(\begin{smallmatrix} 0 & 1 & 0 \\ 0 & 0&0\\ 0 & 0&0\end{smallmatrix}\right)\right).
\end{equation}
This classifies all the tempered, non-supercuspidal $L$-parameters of $\GL_3(F)$.
\section{Classification of Component Groups} \label{mainresultsection} 
The irreducible, tempered representations of $\SL_3(F)$ which are not supercuspidal are constructed by restriction from those of $\GL_3(F)$. Here, restriction means that the given irreducible representation space of $\GL_3(F)$ can be considered as a representation space of $\SL_3(F)$, which is now possibly reducible.

On the other hand, the Langlands dual group of $\SL_3(F)$ is $\PGL_3(\C)$, so an $L$-parameter for $\SL_3(F)$ is a map 
\begin{equation*}
 [\varphi]: W_F\times \SL_2(\C)\to \PGL_3(\C)
\end{equation*}
which arises from projecting the image in $\GL_3(\C)$ of an $L$-parameter $\vp$ of $\GL_3(F)$. Hence, the results of Section \ref{sec-LLC-GL3} are reformulated in what follows to express the $L$-parameters as $L$-homomorphisms $\vp$ instead of Weil-Deligne representations $(\rho,N)$. 

We proceed to compute the component groups of the various $L$-parameters on a case-by-case basis. We begin with a lemma on centralizers of matrices.
\begin{lemma} \label{lem-centralizers}
Let $A$ and $B$ be two matrices in $\GL_3(\C)$. Then $[B]\in Z_{\PGL_3(\C)}([A])$
if and only if $AB=\xi BA$ for some $\xi \in \C$ satisfying $\xi^3=1$. 
\end{lemma}
\begin{proof}
Let $[A]$ and $[B]$ denote the cosets in $\PGL_3(\C)$ represented by the matrices $A,B\in\GL_3(\C)$. Clearly, if $B\in Z_{\GL_3(\C)}(A)$, then $[B]\in Z_{\PGL_3(\C)}([A])$. On the other hand, 
\begin{equation*}
\begin{split}
[B]\in Z_{\PGL_3(\C)}([A]) &\iff [A][B]=[B][A] \\
&\iff [A][B][A^{-1}][B^{-1}]=[I_3].
\end{split}
\end{equation*}
Hence, $ABA^{-1}B^{-1}\in Z(\GL_3(\C))$. Note that the elements of $Z(\GL_3(\C))$ are of the form $zI_3$ for some $z\in \C^\times$, which implies that there exists $z\in \C^\times$ such that $ABA^{-1}B^{-1}=zI_3$. Taking determinants, it is clear that $z^3=1$. Hence, 
\begin{equation*}
 Z_{\PGL_3(\C)}([A])=\left\{[B]\in \PGL_3(\C): AB=\xi BA \text{ for some third root of unity $\xi$ }\right\}. \qedhere
\end{equation*}
\end{proof}
\subsection{Unitary Principal Series $L$-parameters of $\SL_3(F)$} \label{subsec-ups}
Let $\vp$ be an $L$-parameter of $\GL_3(F)$ corresponding to an irreducible, unitary principal series representation. We study the projection of $\mathrm{Im}(\vp)$ to $\PGL_3(\C)$. 
\begin{lemma} \label{lemma-prin-1}
 Let $A\in\GL_3(\C)$ be a matrix of the form 
 \begin{equation*}
 A= \begin{pmatrix}
a_{1} & 0 & 0 \\
 0 &a_2 & 0\\
 0 & 0 &1 \\
\end{pmatrix}
\end{equation*}
where $a_1, a_2$, and 1 are pairwise distinct. The set of matrices $X\in\GL_3(\C)$ that commute with the matrix $A$ is given by 
 \begin{equation*}
 \left\{\begin{pmatrix}x_{11}& 0&0 \\ 0& x_{22}&0\\0&0&x_{33}\end{pmatrix}:x_{ii}\in\C^\times \text{for }i=1,2,3\right\}. 
\end{equation*}
\end{lemma}
\begin{proof}
A matrix $X\in \GL_3(\C)$
commutes with $A$ whenever 
\begin{equation*}
M=XA-AX=
\begin{pmatrix}
0 & (a_2-a_1)x_{12} & (1-a_1)x_{13} \\
(a_1-a_2)x_{21} &0 & (1-a_2)x_{23} \\
(a_1-1)x_{31} & (a_2-1) x_{32}&0 \\
\end{pmatrix} =0.
\end{equation*}
Since $a_1, a_2$, and 1 are pairwise distinct, the only solution to $M=0$ is 
\begin{equation*}
 x_{12}=x_{13}=x_{21}=x_{23}=x_{31}=x_{32}=0. \qedhere
\end{equation*}
\end{proof}

\begin{lemma} \label{lemma-prin-2}
 Let $A\in\GL_3(\C)$ be a matrix of the form 
 \begin{equation*}
 A= \begin{pmatrix}
a_{1} & 0 & 0 \\
 0 &a_2 & 0\\
 0 & 0 &1 \\
\end{pmatrix}
\end{equation*}
where $a_1, a_2$, and 1 are pairwise distinct. For $k\in \{1,2\}$, let $\xi_k=e^{\frac{2k\pi i}{3}}$ and let 
\[\Delta_k= \{X\in\GL_3(\C):
 XA-\xi_k AX=0\}.\]
\begin{enumerate}
 \item[(i)] If $\{a_1,a_2\}\neq \{\xi_1,\xi_2\}$, then $\Delta_1=\Delta_2=\emptyset$.
 \item[(ii)] If $a_1=\xi_{1}$ and $a_2=\xi_2$, then
\[\Delta_1=\left\{\begin{pmatrix}
0 & x_{12} & 0 \\
 0 &0 & x_{23} \\
 x_{31} & 0 & 0 \\
\end{pmatrix}:x_{ij}\in \C^\times \right\} \quad\text{and}\quad \Delta_2=\left\{\begin{pmatrix}
0 & 0 & x_{13} \\
x_{21} &0 & 0 \\
0 & x_{32} & 0 \\
\end{pmatrix}:x_{ij}\in \C^\times \right\}. \]
\item[(iii)] If $a_1=\xi_{2}$ and $a_2=\xi_1$, then
\[\Delta_1=\left\{\begin{pmatrix}
0 & 0 & x_{13} \\
x_{21} &0 & 0 \\
0 & x_{32} & 0 \\
\end{pmatrix}:x_{ij}\in \C^\times \right\} \quad\text{and}\quad \Delta_2=\left\{\begin{pmatrix}
0 & x_{12} & 0 \\
 0 &0 & x_{23} \\
 x_{31} & 0 & 0 \\
\end{pmatrix}:x_{ij}\in \C^\times \right\}. \]
\end{enumerate}
\end{lemma}
\begin{proof}
Let $X \in \GL_3(\C)$, $k\in\{1,2\}$, and
 \begin{equation*}
 M=XA-\xi_k AX=\begin{pmatrix}
a_1(1-\xi_k)x_{11} & (a_2-\xi_k a_1)x_{12} & (1-\xi_k a_1) x_{13}\\
(a_1-\xi_k a_2)x_{21} &a_2(1-\xi_k)  x_{22}& (1-\xi_k a_2) x_{23}\\
(a_1-\xi_k) x_{31}&(a_2-\xi_k)  x_{32}& (1-\xi_k)x_{33} 
\end{pmatrix}.
 \end{equation*}
Since $a_1a_2\neq 0$ and $1-\xi_k\neq 0$, the solution to $m_{11}=m_{22}=m_{33}=0$ is $x_{11}=x_{22}=x_{33}=0$.

To prove part $(i)$, suppose first that $a_1\neq \xi_k$ and $a_2\neq \xi_k$. The solution to $m_{3,i}=0$, for $i=1,2$, is $x_{3,i}=0$. Hence, the third row of the matrix $X$ is the zero vector, which implies that $X\notin \GL_3(\C)$. Suppose next that $a_1=\xi_k$ and $a_2\neq \xi_k^2$. In this case, the solution to $m_{13}=m_{23}=0$ is $x_{13}=x_{23}=0$, and hence $X\notin \GL_3(\C)$. We get the same result in the case $a_2=\xi_k$ and $a_1\neq \xi_k^2$. Therefore, if $\{a_1,a_2\}\neq \{\xi_1,\xi_2\}$, then $\Delta_k=\emptyset$. This proves part $(i)$.

We now prove $(ii)$ and $(iii)$ simultaneously. If $a_1=\xi_k$ and $a_2=\xi_k^2$, then we have $m_{12}=m_{23}=m_{31}=0$ for any values of $x_{12},x_{23},$ and $x_{31}$. Moreover, in this case, the solution to $m_{13}=m_{21}=m_{32}=0$ is $x_{13}=x_{21}=x_{32}=0$. On the other hand, if $a_1=\xi_k^2$ and $a_2=\xi_k$, then we have $m_{13}=m_{21}=m_{32}=0$ for any values of $x_{13},x_{21},$ and $x_{32}$, and the solution to $m_{12}=m_{23}=m_{31}=0$ is $x_{12}=x_{23}=x_{31}=0$. This proves parts $(ii)$ and $(iii)$. 
\end{proof}
\begin{proposition} \label{ups}
Let $\vp$ be an $L$-parameter of $\GL_3(F)$ corresponding to an irreducible, unitary principal series representation. The component group of $[\vp]$ is 
\begin{equation*}
 \mathcal{S}_{[\vp]}\cong 1~~\text{ or }~~\Z/3\Z.
\end{equation*}
\end{proposition} 
\begin{proof}
Recall from \eqref{Langlands parameter 1} that the irreducible, unitary principal series representations of $\GL_3(F)$ correspond to $L$-parameters of the form 
\begin{equation*}
 \left((\chi_1\circ \Art)\oplus (\chi_2\circ \Art)\oplus (\chi_3\circ \Art),\left(\begin{smallmatrix} 0 & 0 & 0\\ 0&0 & 0\\ 0&0 & 0\end{smallmatrix}\right)\right),
\end{equation*}
where $\chi_i, i=1,2,3$, are unitary characters of $F^\times$ and $\chi_i\chi_j^{-1}\neq |\cdot|$ for all $1\leq i,j\leq 3$. 

Suppose that $\vp_i:W_F\times \SL_2(\C)\to \GL_1(\C)$, $i=1,2,3$, are $L$-parameters for $\GL_1(F)$. It is easy to see that $\vp_i$ must be trivial on $\SL_2(\C)$. If $\vp_i(w)=\chi_i\circ \Art (w)$ for every $w\in W_F$, the $L$-parameter (\ref{Langlands parameter 1}) can be written as
\begin{align*}
\vp: W_F &\longrightarrow \GL_1(\C) \times \GL_1(\C) \times \GL_1(\C) \longrightarrow \GL_3(\C) \\
{\quad\,\, w} &\longmapsto (\varphi_1(w), \varphi_2(w), \varphi_3(w)) \longmapsto
\left(\begin{smallmatrix}
\varphi_1(w) & 0 & 0 \\
0 & \varphi_2(w) & 0 \\
0 & 0 & \varphi_3(w)
\end{smallmatrix}\right)
\end{align*}
where the $\vp_i$ are pairwise distinct for all $1\leq i\leq 3$.

Projecting $\mathrm{Im}(\vp)$ to $\PGL_3(\C)$, it is clear that 
\begin{equation*}
 \mathrm{Im}([\vp])= \left\{ \left[ \begin{pmatrix}
\frac{\varphi_1(w)}{\varphi_3(w)} & 0 & 0 \\
0 & \frac{\varphi_2(w)}{\varphi_3(w)} & 0 \\
0 & 0 & 1
\end{pmatrix}\right] \in \PGL_3(\C) : w\in W_F \right\}
\end{equation*}
where $[\vp]= {pr} \circ \vp$. Thus a typical element of $ \mathrm{Im}([\vp])$ has the form $[A]=\left[\left(\begin{smallmatrix}
a_{1} & 0 & 0 \\
 0 &a_2 & 0\\
 0 & 0 &1 \\
\end{smallmatrix}\right)\right] \in \PGL_3(\C)$ for some $a_1,a_2\in \mathbb{C}^\times$, where $a_1,a_2,1$ are pairwise distinct. By Lemma \ref{lem-centralizers}, $[X]\in Z_{\PGL_3(\C)}(\mathrm{Im}[\vp])$ if and only if it satisfies the equation $AX-\xi XA=0$, where $\xi$ is a third root of unity. 
Thus, $Z_{\PGL_3(\C)}(\mathrm{Im}[\vp])=\Delta_{0}\cup \Delta_{1} \cup \Delta_{2}$, where $\Delta_{k}$, $0\leq k\leq 2$ denotes the image in $\PGL_3(\C)$ of the set of all solutions $X\in \GL_3(\C)$ to the equation $AX-e^{\frac{2k\pi i}{3} } XA=0$.
The sets $\Delta_k$ depend on the values of $a_i$ for $i=1$ or $2$ as follows:
\begin{enumerate}[\text{Case} (i)]
 \item If $a_1^3\neq 1$ or $a_2^3\neq 1$, then by Lemma~\ref{lemma-prin-2}(i), we have $\Delta_{l}=\emptyset$ for $l=1,2$. Also, by Lemma~\ref{lemma-prin-1}, we have 
 \begin{equation*}
 \Delta_{0}= \left\{\left[\begin{pmatrix}x_{11}& 0&0 \\ 0& x_{22}&0\\0&0&x_{33}\end{pmatrix}\right]:x_{ii}\in\C^\times\right\}. 
\end{equation*}
 Hence, 
\[
\mathcal{S}_{[\vp]}=Z_{\PGL_3(\C)}(\mathrm{Im}([\vp]))/Z_{\PGL_3(\C)}(\mathrm{Im}([\vp]))^{\circ}= \Delta_{0}/ \Delta_{0}=1.
\]
 \item If $a_i^3= 1$ for $i=1$ and $2$, then since $a_1,a_2$, and 1 are pairwise distinct, we must have either $a_1=\xi_1 a_2$ or $a_1=\xi_2 a_2$, where $\xi_k=e^{\frac{2k\pi i}{3}}$. By Lemma \ref{lemma-prin-1} we have
 \begin{equation*}
 \Delta_{0}= \left\{\left[\begin{pmatrix}x_{11}& 0&0 \\ 0& x_{22}&0\\0&0&x_{33}\end{pmatrix}\right]:x_{ii}\in\C^\times\right\},
\end{equation*} 
and by Lemma~\ref{lemma-prin-2}(ii) and (iii) we have
\begin{equation*}
 \Delta_{1} \cup \Delta_2=\left\{\left[\begin{pmatrix}0&x_{12}& 0 \\ 0&0&x_{23}\\x_{31}&0&0\end{pmatrix}\right]:x_{ij}\in\C^\times \right\}\cup \left\{\left[\begin{pmatrix}0& 0&x_{13} \\ x_{21}&0&0\\0&x_{32}&0\end{pmatrix}\right]:x_{ij}\in\C^\times\right\}. 
\end{equation*}
Therefore, \[\mathcal{S}_{[\vp]}=Z_{\PGL_3(\C)}(\mathrm{Im}([\vp]))/Z_{\PGL_3(\C)}(\mathrm{Im}([\vp]))^{\circ}\cong \Z/3\Z. \qedhere \]
\end{enumerate}
\end{proof}
\subsection{Non-Supercuspidal Discrete $L$-parameter of $\SL_3(F)$} \label{subsec-st}
\begin{lemma} \label{lemma-st}
Let $S$ be the set defined by
 \begin{equation*} \label{matrix-A}
S=\left\{ \begin{pmatrix}
a^{2} & ab & b^{2} \\
2ac & ad+bc & 2bd \\
c^{2} & cd & d^{2}
\end{pmatrix} \in \GL_3(\C): ad-bc=1\right\}.
\end{equation*}
For $k\in \{0,1,2\}$, let $\xi_k=e^{\frac{2k\pi i}{3}}$ and 
\begin{equation*}
 \Delta_k=\left\{X\in \GL_3(\C): XA-\xi_k AX=0 \text{ for all } A \in S \right\}.
\end{equation*}
We have the following:
\begin{enumerate}[(i)] 
\item $\Delta_1=\Delta_2=\emptyset$.
 \item $\Delta_0=\left\{xI_3:x\in\C^\times\right\}$.
\end{enumerate}
\end{lemma}
\begin{proof}
Consider a subset $S_1\subset S$ defined by 
\begin{equation*}
 S_1= \left\{\begin{pmatrix}
a^{2} & 0 & 0 \\
0 & 1 & 0 \\
0 & 0 & a^{-2}
\end{pmatrix} : a\in \C^{\times}\right\}.
\end{equation*}
Note that $B=\left(\begin{smallmatrix}
a^{2} & & \\
 & 1 & \\
 & & a^{-2}
\end{smallmatrix}\right) \in S_1$ is similar to the matrix $C=\left(\begin{smallmatrix}
a^{2} & & \\
 & a^{-2} & \\
 & & 1
\end{smallmatrix}\right)$. It follows from Lemma~\ref{lemma-prin-2}(i) that the equation $XC-\xi_kCX=0$ with $k\in\{1,2\}$ does not have a solution $X\in \GL_3(\C)$ if $a$ is not a third root of unity. We conclude that the set
\begin{equation*}
\left\{X\in \GL_3(\C): XB-\xi_k BX=0 \text{ for all } B \in S_1 \right\}
\end{equation*}
is empty. This proves part $(i)$ since $S_1\subset S$. To compute $\Delta_0$, consider first
\begin{equation*}
XB-BX=\begin{pmatrix}
0 & (1- a^2)x_{12} & (\frac{1}{a^{2}}-a^2)x_{13} \\
(a^{2}-1)x_{21} & 0 & (\frac{1}{a^{2}}-1)x_{23} \\
(a^{2}-\frac{1}{a^2})x_{31} & (1-\frac{1}{a^2}) x_{32}& 0 \end{pmatrix}.
\end{equation*}
Since the matrix above is zero for any choice of $a\in \C^{\times}$ if and only if $X$ is diagonal, we deduce that every element of $\Delta_0$ is diagonal. Finally, for $X\in \GL_3(\C)$ diagonal and $A\in S$ we have
 \begin{equation*}
 XA-AX=
\begin{pmatrix}
 0 & ab(x_{11}-x_{22}) & b^{2}(x_{11}-x_{33}) \\
 2ac(x_{22}-x_{11}) & 0 & 2bd(x_{22}-x_{33}) \\
 c^{2}(x_{33}-x_{11}) & cd(x_{33}-x_{22}) & 0
 \end{pmatrix},
 \end{equation*}
which is zero for any choice of $a,b,c,d\in\C$ if and only if $x_{11}=x_{22}=x_{33}$. This proves part $(ii)$.
\end{proof}
\begin{proposition} \label{pro-st}
 The component group of the non-supercuspidal, discrete $L$-parameter of $\SL_3(F)$ is trivial.
\end{proposition}
\begin{proof}
Recall from \eqref{Langlands parameter 2} that the $L$-parameter corresponding to the twisted Steinberg representation $St_3(\chi)$ is
\begin{equation*}
 \left((\chi\circ \Art)\oplus (\mathbbm{1}\circ\Art)\oplus (\chi^{-1}\circ \Art),\left(\begin{smallmatrix} 0 & 1 & 0\\ 0&0 & 1\\ 0&0 & 0\end{smallmatrix}\right)\right).
\end{equation*}
As an $L$-homomorphism, $\vp$ is trivial on $W_F$ and is the unique, irreducible 3-dimensional representation on $\SL_2(\mathbb{C})$. Explicitly, 
\begin{align*}
 \nonumber \varphi: W_F\times \SL_2(\C)&\to \GL_3(\C) \\
 \left(w,\begin{pmatrix}a & b\\ c& d\end{pmatrix}\right) &\mapsto \begin{pmatrix}
a^{2} & ab & b^{2} \\
2ac & ad+bc & 2bd \\
c^{2} & cd & d^{2}
\end{pmatrix}.
\end{align*}
Projecting $\mathrm{Im}(\vp)$ to $\PGL_3(\C)$, we have an $L$-parameter $[\vp]$ with \begin{equation*}
 \mathrm{Im}([\vp])=\left\{\left[ \begin{pmatrix}
a^{2} & ab & b^{2} \\
2ac & ad+bc & 2bd \\
c^{2} & cd & d^{2}
\end{pmatrix}\right]\in \PGL_3(\C):ad-bc=1\right\}.
\end{equation*}
Since $
 \left\{xI_3:x\in\C^\times\right\} \subset Z(\GL_3(\C))$, it follows from Lemma~\ref{lemma-st} that $Z_{\PGL_3(\C)}(\text{Im}([\varphi]))= \{1\}$. Hence $S_{[\varphi]}=1$.
\end{proof}
\subsection{Non-Discrete Tempered $L$-parameters of $\SL_3(F)$} \label{subsec-nondisc}
 Let $\sigma$ be a discrete series representation of $\GL_2(F)$ and let $\vp_{\sigma}:W_F\times \SL_2(\C)\to\GL_2(\C)$ denote the corresponding $L$-parameter. 
Let $\chi$ be a unitary character of $F^\times$ and let $\vp_\chi:W_F\times\SL_2(\C)\to \GL_1(\C)$ denote the corresponding $L$-parameter. The $L$-parameter corresponding to $\displaystyle \Ind_{P_{2,1}}^{\GL_3(F)}(\sigma \otimes \chi)$ is 
\begin{align*}
 \nonumber \vp:W_F\times \SL_2(\C)&\to \GL_2(\C)\times \GL_1(\C) \subset \GL_3(\C) \\
 \left(w,\begin{pmatrix}a & b\\ c& d\end{pmatrix}\right) &\mapsto \begin{pmatrix} \vp_{\sigma}\left(w,\left(\begin{smallmatrix}a & b\\ c& d\end{smallmatrix}\right)\right) & 0\\0&\vp_{\chi} \left(w,\left(\begin{smallmatrix}a & b\\ c& d\end{smallmatrix}\right)\right) \end{pmatrix}.
\end{align*}
The $L$-parameter $\vp_{\sigma}$ could be dihedral or exceptional, or the representation $\sigma$ could be the Steinberg representation $St_2$. We study each case separately. 

\subsubsection{Dihedral Supercuspidal $L$-parameter $\vp_{\sigma}$}\label{dihedralsubsection}
\begin{lemma} \label{lem-dihedral-1}
Fix $a\in\C^\times$ and let
 \begin{equation*} 
 S=\left\{\begin{pmatrix}
0 & 1 & 0 \\
a & 0 & 0 \\
0 & 0 & c
\end{pmatrix}: c \in \C^\times\right\}.
\end{equation*}
The set of matrices $X\in\GL_3(\C)$ which satisfy the equation $XA-\xi AX=0$ for all $A\in S,$ where $\xi$ is a third root of unity, is given by
\begin{equation*}
 \left\{ \begin{pmatrix}
x_{11} & x_{12} & 0 \\
ax_{12} & x_{11} & 0 \\
0 & 0 & x_{33} \\
\end{pmatrix} : x_{ij}\in\C\right\}\cap \GL_3(\C). 
\end{equation*}
\end{lemma}
\begin{proof}
 Let $X \in \GL_3(\C)$, $c\in \C^\times$, and
 \begin{equation*}
 M=X\begin{pmatrix}
0 & 1 & 0 \\
a & 0 & 0 \\
0 & 0 & c
\end{pmatrix}-\xi \begin{pmatrix}
0 & 1 & 0 \\
a & 0 & 0 \\
0 & 0 & c
\end{pmatrix}X=\begin{pmatrix} ax_{12}-\xi x_{21} & x_{11}-\xi x_{22} & cx_{13}-\xi x_{23} \\ax_{22}-a\xi x_{11} & x_{21}-a\xi x_{12} & cx_{23}- a\xi x_{13} \\
ax_{32}-c\xi x_{31}& x_{31}-c\xi x_{32} & c(1-\xi)x_{33} 
\end{pmatrix}.
 \end{equation*}
First, consider $\xi=1$. The solution to $m_{11}=m_{22}=0$ is given by $x_{21}=ax_{12}$, and the solution to $m_{12}=m_{21}=0$ is $x_{11}=x_{22}$. Also, whenever $c\neq a^2$, we must have $x_{31}=x_{32}=0$ in order for the remaining entries of $M$ to vanish. We conclude that the solution to $M=0$ is given by
\begin{equation*}
 X=\begin{pmatrix}
x_{11} & x_{12} & 0 \\
ax_{12} & x_{11} & 0 \\
0 & 0 & x_{33} \\
\end{pmatrix}.
\end{equation*}
Next, consider $\xi\in \{e^{\frac{2}{3}\pi i},e^{\frac{4}{3}\pi i}\}$. The solution to $m_{11}=m_{22}=0$ is $x_{12}=x_{21}=0$, and the solution to $m_{12}=m_{21}=0$ is $x_{11}=x_{22}=0$. Therefore, $M=0$ has no solution $X\in \GL_3(\C)$. 
\end{proof}
\begin{proposition} \label{pro-dihedral}
 Let $\vp_{\sigma}$ be a dihedral supercuspidal $L$-parameter of $\GL_2(F)$ and $\chi$ be a unitary character of $F^\times$. Let $\vp$ be an $L$-parameter of $\GL_3(F)$ corresponding to the representation $\displaystyle \Ind_{P_{2,1}}^{\GL_3(F)}(\sigma \otimes \chi)$. The component group of $[\vp]$ is trivial.
\end{proposition}
\begin{proof}
By Proposition \ref{prop-supercuspidal}, $[\mathrm{Im}(\vp_\sigma)]\subset \PGL_2(\C)$ is generated by the matrices 
\begin{equation*}
\left[\begin{pmatrix}
0 & 1 \\
1 & 0\\
\end{pmatrix}\right] \quad \text{and} \quad 
\left[\begin{pmatrix}
-1 & 0 \\
0 & 1\\
\end{pmatrix} \right].
\end{equation*}
In this case, $\mathrm{Im}[\vp]$ is generated by the matrices 
\begin{equation*}
 A= \left[\begin{pmatrix}
0 & 1 & 0 \\
1 & 0 & 0\\
0 & 0 & c
\end{pmatrix}\right] \quad \text{and} \quad B=\left[\begin{pmatrix}
-1 & 0 & 0\\
0 & 1 & 0 \\
0 & 0 & c
\end{pmatrix} \right] 
\end{equation*}
in $\PGL_3(\C)$. 
Note that $c \neq 1$ as $c=\vp_\chi (w), w \in W_F$.
By Lemma \ref{lem-dihedral-1} with $a=1$, we have
\begin{equation*}
 Z_{\PGL_3(\C)}(A)= \left\{ [X]: X=\begin{pmatrix}
x_{11} & x_{12} & 0 \\
x_{12} & x_{11} & 0 \\
0 & 0 & x_{33} \\
\end{pmatrix} \in\GL_3(\C) \right\}.
\end{equation*}
Moreover, by Lemma~\ref{lemma-prin-1} and Lemma~\ref{lemma-prin-2}, we have 
 \begin{equation*}
 Z_{\PGL_3(\C)}(B)= \left\{\left[\begin{pmatrix}x_{11}& 0&0 \\ 0& x_{22}&0\\0&0&x_{33}\end{pmatrix}\right]:x_{ii}\in\C^\times \text{ for }i=1,2,3\right\}.
\end{equation*}
Therefore, the set $Z_{\PGL_3(\C)}(\text{Im}([\varphi]))= Z_{\PGL_3(\C)}(A)\cap Z_{\PGL_3(\C)}(B)$ is given by
\begin{equation*}
 Z_{\PGL_3(\C)}(\text{Im}([\varphi]))= \left\{\left[\begin{pmatrix}x_{11}& 0&0 \\ 0& x_{11}&0\\0&0&x_{22}\end{pmatrix}\right]:x_{ii}\in\C^\times \text{ for }i=1,2\right\}.
\end{equation*}
We conclude that $\mathcal{S}_{[\vp]}$ is trivial. 
\end{proof}
 
\subsubsection{Exceptional supercuspidal $L$-parameter $\vp_{\sigma}$} \label{exceptionalparametersection}

\begin{lemma} \label{lem-tetrahedral-2}
Let $c\in \C^\times$ satisfy $c\neq 1$, and let 
 \begin{equation*} 
 A= \begin{pmatrix}
1 & 1 & 0 \\
0 & 1 & 0 \\
0 & 0 & c
\end{pmatrix}\in \GL_3(\C).
\end{equation*}
The set of matrices $X\in\GL_3(\C)$ satisfying the equation $XA-\xi AX=0$ for some third root of unity $\xi \in \C$ is given by
\begin{equation*}
 \left\{ \begin{pmatrix}
x_{11} & x_{12} & 0 \\
0 & x_{11} & 0 \\
0 & 0 & x_{33} \\
\end{pmatrix} : x_{ii}\in\C^\times, x_{12} \in \C, \text{ for }1\leq i\leq 3 \right\}.
\end{equation*}
\end{lemma}
\begin{proof}
 Let $X \in \GL_3(\C)$, and let 
\[M=XA-\xi AX=
\begin{pmatrix}
 (1-\xi) x_{11}-\xi x_{21} & x_{11}+(1-\xi) x_{12}-\xi x_{22} & (c-\xi)x_{13}-\xi x_{23} \\
 (1-\xi) x_{21} & x_{21}+(1-\xi) x_{22} &  (c-\xi) x_{23}\\
 (1-c \xi )x_{31} & x_{31}+(1-c \xi) x_{32} & c (1-\xi) x_{33} \\
\end{pmatrix}.\]
If $\xi\neq 1$ then the solution to $m_{11}=m_{12}=m_{21}=m_{22}=0$ is $x_{11}=x_{12}=x_{21}=x_{22}=0$, hence $M=0$ has no solution with $X\in \GL_3(\C)$. If $\xi=1$ then the solution to $m_{13}=m_{23}=m_{31}=m_{32}=0$ is $x_{13}=x_{23}=x_{31}=x_{32}=0$, and the solution to $m_{11}=m_{12}=m_{22}=0$ is given by $x_{21}=0$, $x_{11}=x_{22}$. The lemma follows.
\end{proof}
\begin{proposition} \label{pro-exceptional}
 Let $\sigma$ be an exceptional supercuspidal representation of $\GL_2(F)$ and $\chi$ be a unitary character of $F^\times$. Let $\vp$ be an $L$-parameter of $\GL_3(F)$ corresponding to the representation $\displaystyle \Ind_{P_{2,1}}^{\GL_3(F)}(\sigma \otimes \chi)$. The component group of $[\vp]$ is trivial.
\end{proposition}
\begin{proof}
By Proposition \ref{prop-supercuspidal}, $[\mathrm{Im}(\vp_\sigma)]\subset \PGL_2(\C)$ is either the tetrahedral group or the octahedral group. Consider first the case when $[\mathrm{Im}(\vp_\sigma)]$ is the tetrahedral group, which is isomorphic to the alternating group 
\begin{equation*}
 A_4=\langle a,b\mid a^2=b^3=(ab)^3=1\rangle
\end{equation*}
generated by the matrices 
\begin{equation*}
\left[\begin{pmatrix}
0 & 1 \\
2 & 0\\
\end{pmatrix}\right] \text{ and } 
\left[\begin{pmatrix}
1 & 1 \\
0 & 1\\
\end{pmatrix}\right] \in \PGL_2(\C).
\end{equation*}
In this case, $\mathrm{Im}[\vp]$ is generated by the matrices 
\begin{equation*}
 A= \left[\begin{pmatrix}
0 & 1 & 0 \\
2 & 0 & 0\\
0 & 0 & c
\end{pmatrix}\right] \quad \text{ and} \quad B=\left[\begin{pmatrix}
1 & 1 & 0\\
0 & 1 & 0 \\
0 & 0 & c
\end{pmatrix} \right] 
\end{equation*}
in $\PGL_3(\C),$ where $c \neq 1$ as $c=\vp_\chi (w)$, $w \in W_F$. 
By Lemma \ref{lem-dihedral-1} with $a=2$, the centralizer in $\PGL_3(\C)$ of the matrix $A$ is
\begin{equation*}
\left\{ [X]:X=\begin{pmatrix}
x_{11} & x_{12} & 0 \\
2x_{12} & x_{11} & 0 \\
0 & 0 & x_{33} \\
\end{pmatrix} \in\GL_3(\C)\right\}.
\end{equation*}
Moreover, by Lemma \ref{lem-tetrahedral-2}, the centralizer in $\PGL_3(\C)$ of the matrix $B$ is 
\begin{equation*} 
 Z_{\PGL_3(\C)}(B)= \left\{ \left[\begin{pmatrix}
x_{11} & x_{12} & 0 \\
0 & x_{11} & 0 \\
0 & 0 & x_{33} \\
\end{pmatrix}\right] : x_{11},x_{33}\in\C^\times, x_{12} \in \C \right\}.
\end{equation*}
Hence, $Z_{\PGL_3(\C)}(\text{Im}([\varphi]))= Z_{\PGL_3(\C)}(A)\cap Z_{\PGL_3(\C)}(B)$ consists of the matrices of the form
\begin{equation*} 
 \left[\begin{pmatrix}
x_{11} & 0 & 0 \\
0 & x_{11} & 0 \\
0&0&1
\end{pmatrix}\right]
\end{equation*}
with $x_{11}\neq 0$, which implies that $\mathcal{S}_{[\vp]}=1$.

Now assume that $[\mathrm{Im}(\vp_\sigma)]$ is the octahedral group, which is isomorphic to the symmetric group $S_4$. Since the alternating group $A_4$ is a subgroup of $S_4$, the centralizer of the image of the $L$-parameter $[\vp]$ is a subset of 
\begin{equation*}
\left\{\left[ \begin{pmatrix}
x_{11} & 0 & 0 \\
0 & x_{11} & 0 \\
0&0&1
\end{pmatrix}\right]: x_{11}\in \C^\times\right\},
\end{equation*}
and hence the component group is again trivial. 
\end{proof}

\subsubsection{Steinberg Representation $\sigma$} 
\begin{lemma} \label{lem-steinberg} 
Let $S$ be the set defined by
 \begin{equation*} 
S=\left\{ \begin{pmatrix}
 a & b &0 \\ c & d &0 \\ 0 & 0 & 1
\end{pmatrix} \in \GL_3(\C): ad-bc\neq 0 \right\}.
\end{equation*}
The set $\Delta$ of matrices $X\in\GL_3(\C)$ that satisfy $XA-\xi AX=0$ for some third roof of unity $\xi$ for all $A\in S$ is given by 
\begin{equation*}
 \left\{\begin{pmatrix}x_{11}& 0&0 \\ 0& x_{11}&0\\0&0&x_{22}\end{pmatrix}:x_{11},x_{22}\in\C^\times\right\}. 
\end{equation*}
\end{lemma}
\begin{proof}
Let $X \in \GL_3(\C)$, $a\in \C^\times$ satisfy $a^2\neq 1$ and $a^3\neq 1$, and
$B=\left(\begin{smallmatrix}
a & 0 & 0 \\
0 & a^{-1} & 0 \\
0 & 0 & 1
\end{smallmatrix}\right)\in S$.
Consider
\begin{equation*}
 M= XB-\xi BX=\begin{pmatrix}
a(1-\xi)x_{11} & (a^{-1}-\xi a)x_{12} & (1-\xi a)x_{13} \\
(a-\xi a^{-1})x_{21} & a^{-1}(1-\xi) x_{22}& (1-\xi a^{-1})x_{23} \\
(a-\xi) x_{31}& (a^{-1}-\xi)x_{32} &(1-\xi) x_{33}\end{pmatrix}.
\end{equation*}
The solution to $m_{31}=m_{32}=0$ is $x_{31}=x_{32}=0$. If $\xi\neq 1$, then the solution to $m_{33}=0$ is $x_{33}=0$, hence $M=0$ has no solution with $X\in \GL_3(\C)$ in this case. Now, assume that $\xi=1$. Then the solution to $m_{12}=m_{13}=m_{21}=m_{23}=0$ is $x_{12}=x_{13}=x_{21}=x_{23}=0$. We deduce that every element of $\Delta$ is diagonal. Finally, for $X\in\GL_3(\C)$ diagonal, we have
\[
 X\begin{pmatrix}
 a & b &0 \\ c & d &0 \\ 0 & 0 & 1
\end{pmatrix}-\begin{pmatrix}
 a & b &0 \\ c & d &0 \\ 0 & 0 & 1
\end{pmatrix}X= \begin{pmatrix} 0 & b(x_{11}-x_{22}) & 0 \\ c(x_{22}-x_{11}) & 0 & 0 \\
0 & 0 & 0
\end{pmatrix},
\]
and this matrix is zero for any choice of $a,b,c,d$ if and only if $x_{11}=x_{22}$. The lemma follows.
 \end{proof}
 
\begin{proposition} \label{pro-st-char}
Let $\sigma$ be the Steinberg representation of $\GL_2(F)$ and $\chi$ be a unitary character of $F^\times$. Let $\vp$ be an $L$-parameter of $\GL_3(F)$ corresponding to the representation $\displaystyle \Ind_{P_{2,1}}^{\GL_3(F)}(\sigma \otimes \chi)$. The component group of $[\vp]$ is trivial.
\end{proposition}
\begin{proof}
 As an $L$-homomorphism, $\vp_\sigma$ is trivial on $W_F$ and is the unique, irreducible 2-dimensional representation on $\SL_2(\mathbb{C})$, which is the identity map on $\SL_2(\mathbb{C})$. Hence, $\vp$ is given by 
\begin{align*}
 \nonumber \varphi: W_F\times \SL_2(\C)&\to \GL_3(\C) \\
 \left(w,\begin{pmatrix}a & b\\ c& d\end{pmatrix}\right) &\mapsto \begin{pmatrix}
 a & b &0 \\ c & d &0 \\ 0 & 0 & \chi\circ \Art (w) 
\end{pmatrix}.
\end{align*}
Let $k=\chi\circ \Art(w)$. Projecting $\mathrm{Im}(\vp)$ to $\PGL_3(\C)$, we have 
\begin{equation*}
 \mathrm{Im}([\vp])=\left\{\left[ \begin{pmatrix}
 a & b &0 \\ c & d &0 \\ 0 & 0 & 1
\end{pmatrix}\right]:ad-bc=\frac{1}{k}\right\}.
\end{equation*}
By Lemma \ref{lem-steinberg}, the centralizer of $\mathrm{Im}([\vp])$ is 
\begin{equation*}
 Z_{\PGL_3(\C)}(\text{Im}([\varphi]))= \left\{\left[\begin{pmatrix}x_{11}& 0&0 \\ 0& x_{11}&0\\0&0&x_{22}\end{pmatrix}\right]:x_{11},x_{22}\in\C^\times\right\} 
\end{equation*}
which implies that $\mathcal{S}_{[\vp]}$ is trivial.
\end{proof}
\section{Closing Remarks} \label{closingremarks}
Unlike $\GL_2$, the classification of supercuspidal $L$-parameters for $\GL_3$ is unavailable. Consequently, the above computations do not apply to irreducible supercuspidal representations of $\SL_3$. One possible direction is to investigate finite subgroups of $\PGL_3(\C)$ and compute their centralizers. Additionally, our matrix computations above are expected to apply to the centralizers of  
certain $L$-parameters of $\SL_n$ with $n >3$, such as Steinberg representations and irreducible proper-parabolic inductions from any combinations of $\GL_2$-supercuspidals, Steinberg representations, and 1-dimensional characters.

\bibliographystyle{alpha}
\bibliography{version}

\end{document}